\documentclass[reqno,12pt]{amsart}
\usepackage{amsmath}
\usepackage{amssymb}
\usepackage{graphicx}
\usepackage{xspace}
\usepackage{bbm}

\numberwithin{equation}{section}

\newcommand\NoBlackBoxes{\global\overfullrule0pt}
\NoBlackBoxes

\newtheorem{definition}{Definition}[section]
\newtheorem{theorem}[definition]{Theorem}
\newtheorem{lemma}[definition]{Lemma}

\theoremstyle{remark}
\newtheorem{rem}[definition]{Remark}
\newtheorem{remark}[definition]{Remark}
\newtheorem{rems}[definition]{Remarks}

\newcommand{\gdw}{\Longleftrightarrow}

\newcommand{\N}{{\mathbb N}}
\newcommand{\R}{{\mathbb R}}
\newcommand{\Z}{{\mathbb Z}}

\newcommand{\E}{{\mathbb E}}
\renewcommand{\P}{{\mathbb P}}

\newcommand{\cov}{\mathrm{Cov}}

\newcommand{\ind}{\mathbbmss{1}}

\newcommand{\X}{\textbf{X}}
\newcommand{\tr}{\mathrm{tr}}
\newcommand{\Y}{\textbf{Y}}
\newcommand{\T}{\textbf{T}}
\newcommand{\W}{\textbf{W}}
\newcommand{\Q}{\textbf{Q}}
\renewcommand{\L}{\textbf{L}}

\newcommand{\grab}{\bigskip \noindent}

\def\eg{e.g.\@\xspace}
\def\ie{i.e.\@\xspace}
\def\iid{i.i.d.\@\xspace}

\def\on{\operatorname}

\title[The spectral density of large covariance matrices with dependence]{On the spectral density of large sample covariance matrices with Markov dependent columns }

\author{Olga Friesen}
\address[Olga Friesen]{Westf\"alische Wilhelms-Universit\"at M\"unster,
Fachbereich Mathematik,
Einsteinstra\ss e 62, 48149 M\"unster, Germany}
\email[Olga Friesen]{olga.friesen@uni-muenster.de}

\author[Matthias L\"owe]{Matthias L\"owe}
\address[Matthias L\"owe]{Westf\"alische Wilhelms-Universit\"at M\"unster,
Fachbereich Mathematik,
Einsteinstra\ss e 62, 48149 M\"unster, Germany}
\email[Matthias L\"owe]{maloewe@math.uni-muenster.de}

\thanks{The research of the first author was supported by DFG through SFB 878 at University of M\"unster}

\date{\today}

\subjclass{60B20, 60F05}

\keywords{random matrix, sample covariance matrix, Pastur-Mar\v{c}enko law, dependent entries, Markov chains}
\begin{document}
\begin{abstract}
We investigate the spectral distribution of large sample covariance matrices with independent columns and entries in the columns that stem from Markov chains. We characterize the limiting spectral densities by their moments. Correspondingly, the proof is based on a moment method.

\end{abstract}

\maketitle

\bibliographystyle{alpha}

\section{Introduction}
Random matrix theory is one of the most active fields in modern probability theory. One of its central questions is the analysis of the spectra of random matrices. There are basically three types of those which have always been in the center of interest: matrices with independent entries (possibly up to symmetry conditions), matrices with invariance properties, \eg orthogonal or unitary invariance, and sample covariance matrices. The present paper is devoted to the study of the latter.

To be more precise, we consider an $m(n)\times n$ matrix $\X_n$ with each column being an independent copy of a stationary sequence whose joint moments satisfy appropriate conditions. We are interested in the sample covariance matrix
\begin{equation*}
\W_n = \frac 1 n \X_n \X_n^T.
\end{equation*}
If $\lambda_1,\ldots,\lambda_{m(n)}$ denote the eigenvalues of $\W_n$ (with multiplicities), we define the empirical spectral measure as
\begin{equation*}
\mu_n:= \frac{1}{m(n)} \sum_{j=1}^{m(n)} \delta_{\lambda_j},
\end{equation*}

where $\delta_{\lambda_j}$ denotes the Dirac measure supported in $\lambda_j$. One of the starting point of modern random matrix theory is the study of the asymptotics of $\mu_n$ as $n\to\infty$ and $\frac{m(n)}{n}\to y\in (0,\infty)$.

Sample covariance matrices are an important tool in statistics, in particular in multivariate statistical
inference. There, test statistics are often defined by the eigenvalues or functionals of sample covariance matrices.
The first and pioneering results for the spectra of sample covariance matrices were obtained under the assumption that all entries of $\X_n$ are
\iid (and not only the columns). Under these conditions Pastur and Mar\v{c}enko \cite{pastur_marcenko} obtained their famous Pastur-Mar\v{c}enko law for the limiting spectral distribution. Extensions of their result can be found in the work of Wachter
\cite{wachter} and Yin \cite{yin}, for a very readable survey paper we refer to \cite{bai_survey}.

For many practical purposes, the assumption that not only the columns but also the entries of the columns are \iid is rather restrictive. Various attempts have thus been made to relax this condition. Silverstein \cite{Silverstein} studied matrices of the form $\X_n= \T_n^{\frac 12}\Y_n$, where $\T_n$ is a non-negative definite matrix and $\Y_n$ consists
of \iid entries. Another attempt to relax the independence assumption was made in Yin and Krishnaiah \cite{yinkrish}. They assumed that the
columns of $\X_n$ are distributed isotropically. Also this result was extended to the case where $\X_n$ is the product of a non-negative definite matrix with an isotropically distributed matrix (see \cite{baiyinkrish}). A very general dependence structure has been considered in \cite{BaiZhou08}. There the limiting measure can be characterized via its Stieltjes transform. Yet another approach was used by Anderson and Zeitouni, who assume joint cumulant summability of the matrix entries, see \cite{zeitouni}, in particular Assumption 2.2 there.

However, all these results do not seem to cover our model in full generality. The aim of the present paper is to derive the limiting spectral distribution of matrices with independent columns, the entries of which are correlated via a Markov process and fulfill suitable conditions. Our main technique is a moment method. The corresponding combinatorial problems will be solved using a graph-theoretical result (see \cite{Kreweras72}).

The rest of the paper is organized in the following way: In the next section we describe the model we are dealing with in detail and formulate our central result. Section \ref{examples} is devoted to examples, while Section \ref{proof} contains the proofs.

\section{The Model and the Main Result}

Suppose that $\left\{a(i,j), i\in\N\right\}$, $j\in\N$, are independent and identically distributed families of real-valued stationary random variables. Define for any $m,n\in\N$ the $m\times n$ matrix $\X_n = (a(i,j))_{1\leq i\leq m, 1\leq j\leq n}$, and assume that $m$ tends to infinity proportionally to $n$, that is
\begin{equation*}
\frac{m}{n} \to y\in (0,\infty), \quad \text{as} \ n\to\infty.
\end{equation*}

Our aim is to impose appropriate conditions on the entries of $\X_n$ to obtain convergence of the expected empirical distribution of the $m\times m$ sample covariance matrix
\begin{equation*}
\W_n = \frac 1 n \X_n \X_n^T,
\end{equation*}

in a situation where the columns of $\X_n$ contain Markov processes. In order to describe such a limit, we will have to control the mixed moments of the entries. Thus, we start with the covariances, and put for any $i,i'\in\N$,
\begin{equation*}
t(i,i') := \cov(a(i,j),a(i',j)),
\end{equation*}

which does not depend on $j\in\N$. For any $m\in\N$, denote by $\T_m = (t(i,i'))_{1\leq i,i'\leq m}$ the covariance matrix of the sequence $\left\{a(i,j), 1\leq i\leq m \right\}$. Assume that for any $k\in\N$, the sequence of the $k$-th moments of the empirical spectral distribution $\nu_m$ of $\T_m$ converges, and put
\begin{equation}
H_k:= \lim_{m\to\infty} \int x^k \ \nu_m(dx) = \lim_{m\to\infty} \frac 1 m \tr\left(\T_m^k\right).
\label{mt}
\end{equation}

To state the remaining conditions, we need to introduce some notation. Hence, for $j=1,\ldots,k$, we denote by $S_{2k}^j$ the set of all permutations $\sigma$ of $\{1,\ldots,2k\}$ such that $\{\sigma(2l-1), \sigma(2l)\}\neq \{2l'-1,2l'\}$ for at least one $l\in\{1,\ldots,j\}$ and any $l'\in\{1,\ldots,k\}$. Now assume that for any $m\in\N$, there is a deterministic $m\times m$ matrix
\begin{equation*}
\T_m'=(t'(i,i'))_{1\leq i,i'\leq m} =(\alpha^{|i-i'|})_{1\leq i,i'\leq m},
\end{equation*}

with $\alpha\in[0,1)$, such that for any $k\in\N$ and $i_1,\ldots,i_{2k}\in\{1,\ldots,m\}$, $i_1\leq\ldots\leq i_{2k}$, \\

\begin{enumerate}
	\item[(A1)] we have
		\begin{align*}
				\E\left[a(i_1,1)\cdot\ldots\cdot a(i_{2k},1)\right] = \prod_{l=1}^{k} t(i_{2l-1} ,i_{2l}) + R(i_1,\ldots,i_{2k}),
		\end{align*}
		where
		\begin{equation*}
				|R(i_1,\ldots,i_{2k})| \leq c(k) \sum_{j=1}^k \sum_{\sigma\in S_{2k}^j} \prod_{l=1}^{j} t'(i_{\sigma(2l-1)} ,i_{\sigma(2l)}),
		\end{equation*}
		with $c(k)\geq 0$ depending only on $k$ and not on $m$,
	\item[]
	\item[(A2)] it holds that
		\begin{equation*}
				\left|\E\left[a(i_1,1)\cdot\ldots\cdot a(i_{2k},1)\right]\right| \leq c(k) \prod_{l=1}^{k} t'(i_{2l-1} ,i_{2l}).
		\end{equation*}
\end{enumerate}

\grab

Note that none of the conditions (A1) and (A2) implies the other. The quality of the estimates depends on the particular choice of the elements $i_1\leq\ldots\leq i_{2k}$. Moreover, (A1) and (A2) entail the existence of moments of all orders. In Section~\ref{examples}, we will present some examples of stationary processes that satisfy the conditions above. In particular, such processes necessarily have exponentially decaying covariances which is the case for many Markov processes. The definition of the sets $S_{2k}^j$ and the estimate in (A2) are basically motivated by the representation of the joint moments of a Gaussian process. Here, Isserlis' theorem \cite{Isserlis} states that
\begin{equation*}
	\E\left[a(i_1,1)\cdot\ldots\cdot a(i_{2k},1)\right] = \prod_{l=1}^{k} t(i_{2l-1} ,i_{2l}) + \sum_{\sigma\in S_{2k}^k} \prod_{l=1}^{k} t(i_{\sigma(2l-1)} ,i_{\sigma(2l)}).
\end{equation*}

It will become obvious that in our situation, condition (A2) is the most restrictive since it basically demands the underlying distribution to be symmetric. Our main result is

\begin{theorem}
Assume that \eqref{mt}, (A1), and (A2) hold. Then, the expected empirical spectral distribution of $\W_n$ converges weakly to a probability measure with $k$-th moment equal to
\begin{equation}
\sum_{s=1}^{k} \ y^{k-s} \ \frac{k!}{s!} \ \sum_{\substack{i_1+\ldots + i_s = k-s+1, \\ i_1 +2i_2+\ldots + si_s=k}} \  \prod_{l=1}^{s} \frac{H_{l}^{i_l}}{i_l!}.
\label{lsdmom}
\end{equation}
\label{main}
\end{theorem}

\begin{rem}
Let $\textbf{Y}_n$ be an $m\times n$ random matrix with \iid entries. In the proof of Theorem 4.1 in \cite{BaiSilv}, we see that the limiting distribution in Theorem~\ref{main} coincides with that of $\frac{1}{n} \textbf{Y}_n \textbf{Y}^T_n \T_m$. This is not surprising, at least when considering Gaussian entries. Indeed, if $\T_m$ is a positive definite matrix with Cholesky decomposition $\T_m=\textbf{L}_m \textbf{L}_m^T$, then $\frac{1}{n} \textbf{Y}_n \textbf{Y}^T_n \T_m$ has the same limiting distribution as $\frac{1}{n} \textbf{L}_m \textbf{Y}_n \textbf{Y}^T_n \textbf{L}_m^T$. Assuming that the entries of $\textbf{Y}_n$ are Gaussian, we see that the matrix $\textbf{L}_m \textbf{Y}_n$ has an independent copy of the same Gaussian Markov process in each column.
\end{rem}

\begin{rem}
The moments in Theorem~\ref{main} satisfy the Carleman condition (cf. \cite{BaiSilv}, Lemma 4.9). Hence, they determine the limiting probability distribution uniquely.
\end{rem}

\begin{rem}[Random Fields]
A natural extension of Theorem \ref{main} is to consider the limiting spectral distribution of matrices
\begin{equation*}
\frac{1}{n} \X_n \Q_n \X_n^T,
\end{equation*}

where $\Q_n:=(q(i,j))_{1\leq i,j\leq n}$ is a deterministic symmetric $n\times n$ matrix. If $\Q_n$ is positive definite, and $\Q_n=\L_n \L_n^T$ for a lower triangular matrix $\L_n$, then the matrix $\Y_n:= \X_n \L_n$ forms a Markov random field on a lattice. Now assume that the conditions of Theorem \ref{main} hold, and for any $k\in\N$, $i_1,\ldots,i_{2k-1}\in\N$, we have $\E[a(i_1,j)\cdot\ldots\cdot a(i_{2k-1},j)] = 0$. Further, put $\Q_n':=(|q(i,j)|)_{1\leq i,j\leq n}$, and suppose that the limit $\lim_{n\to\infty} \frac{1}{n} \tr((\Q_n')^k)$ exists for any $k\in\N$. Define
\begin{equation*}
\tilde{H}_k:= \lim_{n\to\infty} \frac 1 n \tr\left(\Q_n^k\right).
\end{equation*}

One can apply similar techniques as in the proof of Theorem \ref{main} to show that the $k$-th moment of the expected empirical spectral distribution of $\frac{1}{n} \X_n \Q_n \X_n^T$ converges to
\begin{equation*}
\sum_{s=1}^k y^{k-s} k (k-s)! (s-1)! \sum_{\substack{i_1+\ldots + i_s = k-s+1, \\ i_1 +2i_2+\ldots + si_s=k}} \  \prod_{l=1}^{s} \frac{H_{l}^{i_l}}{i_l!} \sum_{\substack{j_1+\ldots + j_{k-s+1} = s, \\ j_1 +2j_2+\ldots + (k-s+1)j_{k-s+1}=k}} \prod_{l=1}^{k-s+1} \frac{\tilde{H}_{l}^{j_l}}{j_l!}.
\end{equation*}

The proof will be given in a separate article.
\end{rem}

\section{Examples}
\label{examples}

\subsection{Independent Random Variables}

Suppose that the families $\left\{a(i,j), i\in\N\right\}$, $j\in\N$, consist of \iid random variables with zero mean and variance $\sigma^2>0$. Moreover, for any $k\in\N$, we assume that $m(2k):=\E[a(i,j)^{2k}]<\infty$, and $\E[a(i,j)^{2k-1}]=0$. In this case, we have $\T_m=\sigma^2 \textbf{I}_m$, implying that
\begin{equation*}
	\frac 1 m \tr\left(\T_m^k\right) = \sigma^{2k}.
\end{equation*}
Consequently, \eqref{mt} holds with $H_k=\sigma^{2k}$. Now choose $\T_m'=\T_m$, and fix $i_1\leq \ldots \leq i_{2k}$. If the value of some $i_l$ occurs only once, then
\begin{equation*}
	\E\left[a(i_1,1)\cdot\ldots\cdot a(i_{2k},1)\right] = 0 = \prod_{l=1}^{k} t(i_{2l-1} ,i_{2l}).
\end{equation*}

If $i_1=i_2<i_3=i_4<\ldots<i_{2k-1}=i_{2k}$, we have that
\begin{equation*}
	\E\left[a(i_1,1)\cdot\ldots\cdot a(i_{2k},1)\right] = \sigma^{2k} = \prod_{l=1}^{k} t(i_{2l-1} ,i_{2l}),
\end{equation*}
implying (A1) in this situation. In any other case, we can conclude that the sum $\sum_{\sigma\in S_{2k}^j} \prod_{l=1}^{j} t'(i_{\sigma(2l-1)} ,i_{\sigma(2l)})$ has at least one constant term. Thus, (A1) holds since the boundedness of the moments and the H\"older inequality yield
\begin{equation*}
	\left|\E\left[a(i_1,j)\cdot\ldots\cdot a(i_{2k},j)\right]\right|\leq m(2k).
\end{equation*}

Since the elements $i_1,\ldots,i_{2k}$ are sorted, we know that the identity $\prod_{l=1}^{k} t(i_{2l-1} ,i_{2l})=0$ entails the existence of an odd number $j$ such that $i_l=\ldots=i_{l+j-1}$ for some $l$, and $i_{l'}\neq i_l$ for any $l'\notin \{l,\ldots,l+j-1\}$. In this case, there is an odd moment that occurs in the product $\E\left[a(i_1,1)\cdot\ldots\cdot a(i_{2k},1)\right]$, implying that this expectation is also equal to zero. If $\prod_{l=1}^{k} t(i_{2l-1} ,i_{2l})\neq 0$, then $\prod_{l=1}^{k} t(i_{2l-1} ,i_{2l})=\sigma^{2k}$, and we can put $c(k):= m(2k)/\sigma^{2k}$ to obtain condition (A2). By Theorem~\ref{main}, we thus have
\begin{align*}
\lim_{n\to\infty} \frac 1 m \E\left[\tr\left(\W_n^k\right)\right]
= \sigma^{2k} \sum_{s=1}^{k} \ y^{k-s} \ \frac{k!}{s!} \ \sum_{\substack{i_1+\ldots + i_s = k-s+1, \\ i_1 +2i_2+\ldots + si_s=k}} \  \prod_{l=1}^{s} \frac{1}{i_l!}.
\end{align*}

Denoting by $\mathrm{NC}^{(i)}(k)$ the set of all non-crossing partitions of $\{1,\ldots,k\}$ with $i$ blocks, Lemma~\ref{noncr} yields
\begin{equation*}
\frac{k!}{s!} \ \sum_{\substack{i_1+\ldots + i_s = k-s+1, \\ i_1 +2i_2+\ldots + si_s=k}} \ \prod_{l=1}^{s} \frac{1}{i_l!} = \# \mathrm{NC}^{(k-s+1)}(k),
\end{equation*}

implying
\begin{align}
\lim_{n\to\infty} \frac 1 m \E\left[\tr\left(\W_n^k\right)\right] = \sigma^{2k} \sum_{i=0}^{k-1} \ y^i \ \# \mathrm{NC}^{(i+1)}(k).
\label{pastur}
\end{align}

It was proven in \cite{Kreweras72}, Corollary 4.1, that
\begin{equation*}
\# \mathrm{NC}^{(i+1)}(k) = \frac{1}{k} \binom{k}{i} \binom{k}{i+1} = \frac{1}{i+1} \binom{k}{i} \binom{k-1}{i}.
\end{equation*}

Substituting this identity in \eqref{pastur}, we exactly obtain the moments of the Mar\v{c}enko-Pastur distribution with parameter $y\in (0,\infty)$ (\cite{BaiSilv}, Lemma 3.1).

\subsection{Stationary Processes on a Finite State Space}

Let $\left\{a(i), i\in\N\right\}$ be a stationary process on a finite state space $S=\{s_1,\ldots,s_N\}$, $N\geq 2$. Denote by $\pi=(\pi_1,\ldots,\pi_N)$ the stationary distribution. Further, suppose that $\E[a(1)]=0$ and, for any $k,l\in\N$, $l\leq k$, $i_1\leq \ldots \leq i_k$,
\begin{equation}
\begin{split}
& \max_{j_1,\ldots,j_k\in\{1,\ldots,N\}} \Big|\P\left(a(i_k)=s_{j_k},\ldots,a(i_l)=s_{j_l} \ | \ a(i_{l-1})=s_{j_{l-1}},\ldots,a(i_1)=s_{j_1}\right) \Big. \\
& \hspace{5.2cm} \Big. - \P\left(a(i_k)=s_{j_k},\ldots,a(i_l)=s_{j_l}\right) \Big|
\leq C \alpha^{i_l-i_{l-1}}.
\end{split}
\label{rate}
\end{equation}

Before continuing, let us consider some concrete examples. On the one hand, aperiodic and irreducible Markov chains satisfy this inequality. On the other hand, we can also consider different Gibbs measures. To be more precise, assume that the joint distribution of the process $\left\{a(i), i\in\Z\right\}$ is a Gibbs measure for some shift-invariant potential $\Phi=\{\phi_A: A\subset \Z, 0<|A|<\infty\}$. Due to \cite{Georgii}, Chapter 8, the estimate in \eqref{rate} holds if
\begin{equation}
\sum_{A: 0\in A} e^{t \mathrm{diam}(A)} (|A|-1) \sup_{\zeta,\eta} |\phi_A(\zeta)-\phi_A(\eta)| < \infty,
\label{expcond}
\end{equation}

for some $t>0$, and Dobrushin's condition is satisfied. Since we consider a shift-invariant potential $\Phi$, the latter is true if
\begin{equation}
\sum_{A: 0\in A} (|A|-1) \sup_{\zeta,\eta} |\phi_A(\zeta)-\phi_A(\eta)| < 2.
\label{dobrushin}
\end{equation}

If we take, for example, a potential with finite range, then \eqref{expcond} is satisfied. If we consider a Gibbs measure with a parameter $\beta>0$, \ie we substitute $\Phi$ by $\Phi_\beta:=\{\beta\phi_A: A\subset \Z, 0<|A|<\infty\}$, then \eqref{dobrushin} holds whenever $\beta$ is small enough. \\

To verify condition (A1), the assumptions made are sufficient. This is also the case for the prove of \eqref{mt}. However, in order to obtain (A2), we have to assume that for any $(i_1,\ldots,i_k)\in\N^k$,
\begin{equation*}
\E[a(i_1)\cdot\ldots\cdot a(i_k)] = 0,
\end{equation*}

whenever $k\in\N$ is odd. In particular, we have that $a(1)$ and $-a(1)$ are equally distributed, implying that $\pi$ is symmetric. \\

Note that for any $k\in\N$, $l\leq k$, $i_1\leq \ldots\leq i_{l-1} \leq i_l<\ldots<i_k$, and $d_l,\ldots,d_k\in\N$, it holds that
\begin{align*}
&\E[a(i_k)^{d_k}\cdot\ldots\cdot a(i_l)^{d_l} \ | \ a(i_{l-1}),\ldots,a(i_1)] \\
& \hspace{1.5cm} = \sum_{j_1,\ldots,j_k=1}^N s_{j_k}^{d_k}\cdot\ldots\cdot s_{j_l}^{d_l} \ \ind_{\{ a(i_{l-1})=s_{j_{l-1}},\ldots,a(i_1)=s_{j_1} \}} \\
& \hspace{3.5cm} \cdot \P\left(a(i_k)=s_{j_k},\ldots,a(i_l)=s_{j_l} \ | \ a(i_{l-1})=s_{j_{l-1}},\ldots,a(i_1)=s_{j_1}\right) .
\end{align*}

Consequently, \eqref{rate} implies
\begin{equation}
\left| \E[a(i_k)^{d_k}\cdot\ldots\cdot a(i_l)^{d_l} \ | \ a(i_{l-1}),\ldots,a(i_1)] - \E[a(i_k)^{d_k}\cdot\ldots\cdot a(i_l)^{d_l}] \right| \leq C \alpha^{i_l-i_{l-1}}.
\label{rate2}
\end{equation}

We want to start with verifying the convergence of the empirical spectral distribution of $\T_m$. Since the considered process is stationary, we can define $R(j)=t(1,j+1)=\E[a(1)a(j+1)]$, $j\in\N_0$, implying that $\T_m = (R(|j-j'|))_{1\leq j,j'\leq m}$. Note that in particular, we have for any $j\in\N_0$,
\begin{equation*}
|R(j)| = \left|\E\left[a(1) \left( \E[a(j+1)|a(1)]-\E[a(j+1)]\right) \right] \right| \leq C \alpha^{j}.
\end{equation*}

Consequently, the covariances are summable, which entails the existence of the spectral density $f:[0,1]\to \R$,
\begin{equation*}
f(x) = \sum_{j\in\Z} e^{2\pi ijx} R(j),
\end{equation*}

where $R(-j):=R(j)$ for any $j\geq 1$. In this case, Szeg\"o's limit theorem yields that the moments of the empirical spectral distribution $\nu_m$ of $\T_m$ converge. To be more precise, we obtain for any $k\in\N$,
\begin{equation*}
\lim_{m\to\infty} \int_{\R} x^k \ \nu_m(dx) = \int_0^1 f(x)^k \ dx.
\end{equation*}

Thus, \eqref{mt} holds. Now put $p=\sqrt[k]{\alpha}$, and $t'(i,j)=p^{|i-j|}$, $1\leq i,j\leq m$. Let $i_1\leq \ldots \leq i_{2k}$. To obtain assumption (A1), it suffices to verify that
\begin{equation}
\E[a(i_1)\cdot\ldots\cdot a(i_{2k})] = \prod_{l=1}^{k} t(i_{2l-1},i_{2l}) + R(i_1,\ldots,i_{2k}),
\label{goal}
\end{equation}

where
\begin{equation}
|R(i_1,\ldots,i_{2k})| \leq C \sum_{l=1}^{k-1} \alpha^{i_{2l+1}-i_{2l}}.
\label{goal2}
\end{equation}

To achieve this aim, first consider $k=1$. In this case, the identity in \eqref{goal} holds with $R(i_1,\ldots,i_{2k})=0$. Now take any arbitrary $k\in\N$ and assume that our statement holds for any $k'<k$. We can apply \eqref{rate2} to obtain
\begin{align}
\begin{split}
& \E[a(i_1)\cdot\ldots\cdot a(i_{2k})] \\
& \quad = \E\Big[a(i_1) a(i_2) \Big(\E[a(i_3)\cdot\ldots\cdot a(i_{2k}) \ | \ a(i_1),a(i_2)] - \E[a(i_3)\cdot\ldots\cdot a(i_{2k})]\Big)\Big] \\
& \qquad + t(i_1,i_2) \ \E[a(i_3)\cdot\ldots\cdot a(i_{2k})] \\
& \quad = \prod_{l=1}^k t(i_{2l-1},i_{2l}) + R(i_1,\ldots,i_{2k}),
\end{split}
\label{cov}
\end{align}

with
\begin{align*}
|R(i_1,\ldots,i_{2k})| \leq C \alpha^{i_3-i_2} + C \sum_{l=2}^{k-1} \alpha^{i_{2l+1}-i_{2l}} = C \sum_{l=1}^{k-1} \alpha^{i_{2l+1}-i_{2l}}.
\end{align*}

Hence, condition (A1) holds. To prove (A2), note that the fact that all odd mixed moments vanish implies that for $l\leq k$, we have
\begin{align*}
& \E[a(i_1)\cdot\ldots\cdot a(i_{2k})] \\
& \hspace{1cm} = \E\Big[a(i_1)\cdot\ldots\cdot a(i_{2l-1}) \Big. \\
& \hspace{2cm} \cdot \Big. \Big(\E[a(i_{2l})\cdot\ldots\cdot a(i_{2k}) \ | \ a(i_1),\ldots, a(i_{2l-1})] - \E[a(i_{2l})\cdot\ldots\cdot a(i_{2k})]\Big)\Big].
\end{align*}

Again, the estimate in \eqref{rate2} ensures that
\begin{equation*}
|\E[a(i_1)\cdot\ldots\cdot a(i_{2k})]| \leq C \alpha^{i_{2l}-i_{2l-1}}.
\end{equation*}

Since this relation holds for any $l=1,\ldots,k$, we can conclude
\begin{equation*}
|\E[a(i_1)\cdot\ldots\cdot a(i_{2k})]| \leq C \alpha^{\max\{i_{2l}-i_{2l-1} : l=1,\ldots,k \}}
\leq C \prod_{l=1}^k p^{i_{2l}-i_{2l-1}}.
\end{equation*}

\subsection{Gaussian Processes}

Assume that for any $j\in\N$, the stationary process $\left\{a(i,j), i\in\N\right\}$ is Gaussian with zero mean. By Isserlis' theorem, assumption (A1) holds. Moreover, if $\left\{a(i,j), i\in\N\right\}$ is additionally a non-degenerated Markov process, we can conclude that $t(i,i') = p^{|i-i'|}$ for some $-1 < p < 1$ and any $i,i'\in\N$. In particular, (A2) is satisfied. The convergence of the empirical spectral distribution of $\T_m$ has been verified in the previous example.

\section{Proof of Theorem \ref{main}}
\label{proof}

The idea of the proof is to use the method of moments. Thus, we want to show that for any $k\in\N$, the $k$-th moment of the expected empirical spectral distribution $\bar{\mu}_n$ of $\W_n$ converges to \eqref{lsdmom} as $n\to\infty$. Our starting point is the identity
\begin{align*}
& \int x^k \ \bar{\mu}_n(dx) = \frac 1 m \E\left[\tr\left(\W_n^k\right)\right] \\
& \quad = \frac {1} {m n^k} \sum_{i_1,\ldots,i_k=1}^m \sum_{j_1,\ldots,j_k = 1}^n \E\left[a(i_1,j_1) a(i_2,j_1) a(i_2,j_2) a(i_3,j_2) \cdot \ldots \cdot a(i_k,j_k) a(i_1,j_k)\right].
\end{align*}

\subsection{The independence of the columns}

Let $k\in \N$ and denote by $\mathcal{P}(k)$ the set of all partitions of $\left\{1,\ldots,k\right\}$. We say that two elements $l,l'\in \left\{1,\ldots,k\right\}$ are \textit{equivalent} with respect to a partition $\pi \in \mathcal{P}(k)$, and write $l\sim_\pi l'$, if $l$ and $l'$ are in the same block of $\pi$. Further, for any fixed $\pi \in \mathcal{P}(k)$, define $S_n(\pi)$ to be the set of all $k$-tuples $(j_1,\ldots,j_k) \in \left\{1,\ldots,n\right\}^k$ such that
\begin{equation*}
j_l = j_{l'} \qquad \gdw \qquad l\sim_\pi l'.
\end{equation*}

In particular, we have that for $(j_1,\ldots,j_k) \in S_n(\pi)$, the entries $a(i_1,j_{l_1}),\ldots,a(i_p,j_{l_p})$ are independent whenever $l_1,\ldots,l_p$ belong to different blocks of $\pi$. Thus, we obtain
\begin{equation*}
\begin{split}
\frac 1 m \E\left[\tr\left(\W_n^k\right)\right]
& = \frac {1} {m n^k} \sum_{\pi\in\mathcal{P}(k)} \sum_{(j_1,\ldots,j_k)\in S_n(\pi)} \sum_{i_1,\ldots,i_k=1}^m \E\bigg[\prod_{l=1}^k a(i_l,j_l) a(i_{l+1},j_l)\bigg] \\
& = \frac {1} {m n^k} \sum_{\pi\in\mathcal{P}(k)} \sum_{(j_1,\ldots,j_k)\in S_n(\pi)} \sum_{i_1,\ldots,i_k=1}^m
\prod_{s=1}^{\#\pi}  \E\bigg[\prod_{l\in B_\pi^{(s)}} a(i_l,j_l) a(i_{l+1},j_l)\bigg],
\end{split}
\end{equation*}

where we cyclically identify $k+1$ with $1$, and write $\#\pi$ for the number of blocks of $\pi$ denoted by $B_\pi^{(1)},\ldots,B_\pi^{(\#\pi)}$. Since we assumed the columns of $\X_n$ to be identically distributed, we conclude
\begin{equation*}
\frac 1 m \E\left[\tr\left(\W_n^k\right)\right] = \frac {1} {m n^k} \sum_{\pi\in\mathcal{P}(k)} \ \# S_n(\pi) \ \sum_{i_1,\ldots,i_k=1}^m
\prod_{s=1}^{\#\pi}  \E\bigg[\prod_{l\in B_\pi^{(s)}} a(i_l,1) a(i_{l+1},1)\bigg].
\end{equation*}

To fix some element in $S_n(\pi)$, we have to choose for each block of $\pi$ one value in $\{1,\ldots,n\}$. By definition, those values are supposed to be distinct, implying $\# S_n(\pi) = n (n-1) \cdot\ldots\cdot (n-\#\pi+1)$. We thus arrive at
\begin{equation}
\lim_{n\to\infty} \frac 1 m \E\left[\tr\left(\W_n^k\right)\right]
= \lim_{n\to\infty} \sum_{\pi\in\mathcal{P}(k)} \frac {1} {m n^{k-\#\pi}} \sum_{i_1,\ldots,i_k=1}^m
\prod_{s=1}^{\#\pi}  \E\bigg[\prod_{l\in B_\pi^{(s)}} a(i_l,1) a(i_{l+1},1)\bigg],
\label{moments}
\end{equation}

if the limits exist.

\subsection{Consistent graphs}

Our final aim is to use assumption (A1) in order to deduce a representation of the expectations in \eqref{moments} with products of the form $t(l_1,l_2) t(l_2,l_3)\cdot\ldots\cdot t(l_{d-1},l_d) t(l_d,l_1)$. This would lead to the occurrence of traces $\tr \left(\T_m^d\right)$. In this case, relation \eqref{mt} could be used to compute the limit in \eqref{moments}. Toward this end, we start with some definitions. Thus, fix $\pi \in \mathcal{P}(k)$ and denote for any $s=1,\ldots,\#\pi$ the \emph{closed blocks} of $\pi$ by
\begin{equation*}
\overline{B}_\pi^{(s)} := B_\pi^{(s)} \cup \left\{l\in \{1,\ldots,k\}: l-1\in B_\pi^{(s)}\right\},
\end{equation*}

where $0$ is identified with $k$. Moreover, define the function $m_\pi: \left\{1,\ldots,k\right\} \to \left\{1,2\right\}$ by
\begin{equation*}
m_\pi(l) = \begin{cases} 1, \qquad \text{if} \ l \ \not\sim_\pi \ l-1, \\ 2, \qquad \text{if} \ l \ \sim_\pi \ l-1. \end{cases}
\end{equation*}

If $l\in \overline{B}_\pi^{(s)}$ for some $s=1,\ldots,\#\pi$, then $m_\pi(l)$ is the \emph{multiplicity} of $l$ in $\overline{B}_\pi^{(s)}$. This definition allows us to write

\begin{equation}
\E\bigg[\prod_{l\in B_\pi^{(s)}} a(i_l,1) a(i_{l+1},1)\bigg] = \E\bigg[\prod_{l\in \overline{B}_\pi^{(s)}} a(i_l,1)^{m_\pi(l)}\bigg].
\label{expect}
\end{equation}

In order to employ the representation of the mixed moments in (A1), it will be necessary to sort the elements $i_l$, $l\in\overline{B}_\pi^{(s)}$, for any $s=1,\ldots,\#\pi$. The idea is to introduce certain graphs on the circle $\{1,\ldots,k\}$ such that an edge between two vertices $v$ and $w$ indicates that both are elements of the same set $\overline{B}_\pi^{(s)}$, and $i_v$ and $i_w$ are neighbors concerning their size. To make this description more precise, we begin with

\begin{definition}[Consistent Graphs]

Let $G$ be an undirected multigraph with vertex set $V(G) = \left\{1,\ldots,k\right\}$ and edge set $E(G) = \left\{e_1,\ldots,e_k\right\}$, and let $f_G: E \to V \cup [V]^2$ be a function assigning to each edge either one or two vertices called the \emph{ends} of the edge. We say that $G$ is \emph{consistent} with $\pi$ if it is possible to decompose $G$ into subgraphs $G_\pi^{(s)}$, $s=1,\ldots,\#\pi$, with vertex sets $\overline{B}_\pi^{(s)}$ and edge sets $E_\pi^{(s)}(G)$, such that \\

\begin{enumerate}
	\item[(C1)] the degree of a vertex $l$ in the subgraph $G_\pi^{(s)}$ equals $m_\pi(l)$,
	\item[(C2)] $E(G)$ is the disjoint union of the sets $E_\pi^{(s)}(G),\ s=1,\ldots,\#\pi$.
\end{enumerate}
\label{def}
\end{definition}

In other words, a $\pi$-consistent graph $G$ can be constructed by connecting the elements of any set $ \overline{B}_\pi^{(s)}$ according to their multiplicities (Figure \ref{graphs}).

\begin{figure}[ht]
  \centering
  \begin{minipage}[b]{4.9 cm}
    \includegraphics[trim = 0cm 5cm 0cm 0cm, width=49mm]{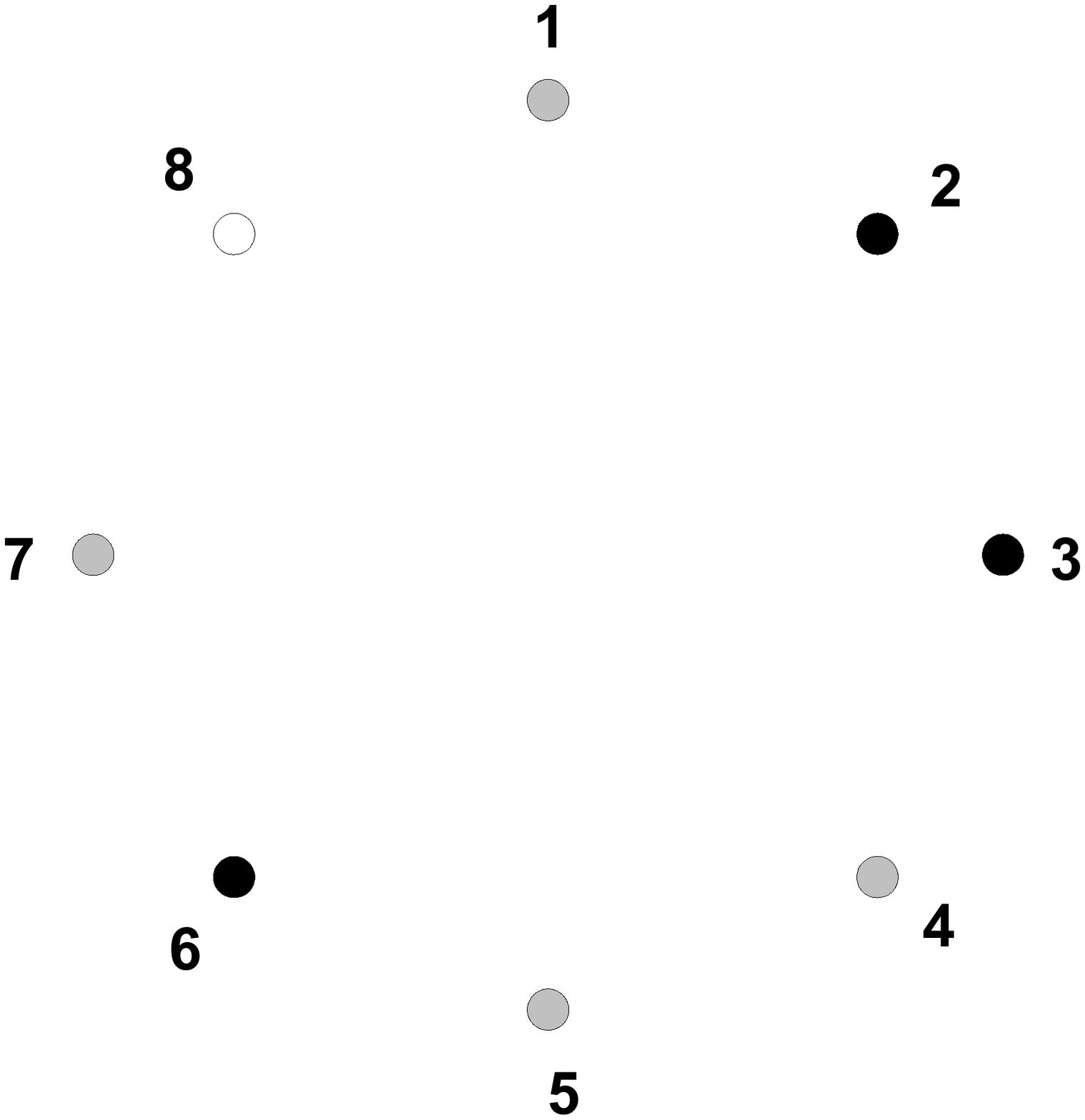}
  \end{minipage}
  \begin{minipage}[b]{4.9 cm}
    \includegraphics[trim = 0cm 5cm 0cm 0cm, width=49mm]{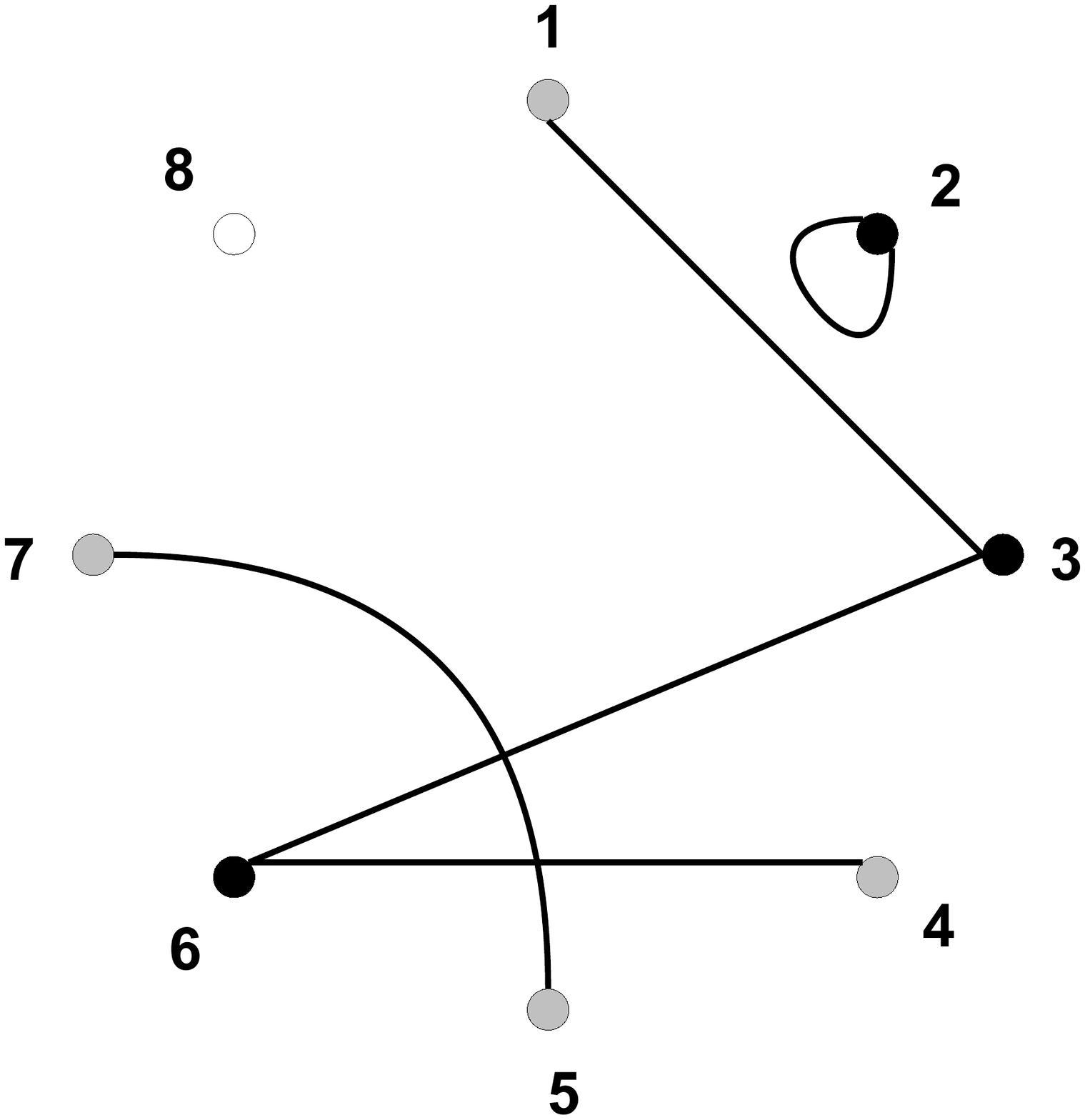}
  \end{minipage}
  \begin{minipage}[b]{4.9 cm}
    \includegraphics[trim = 0cm 5cm 0cm 0cm, width=49mm]{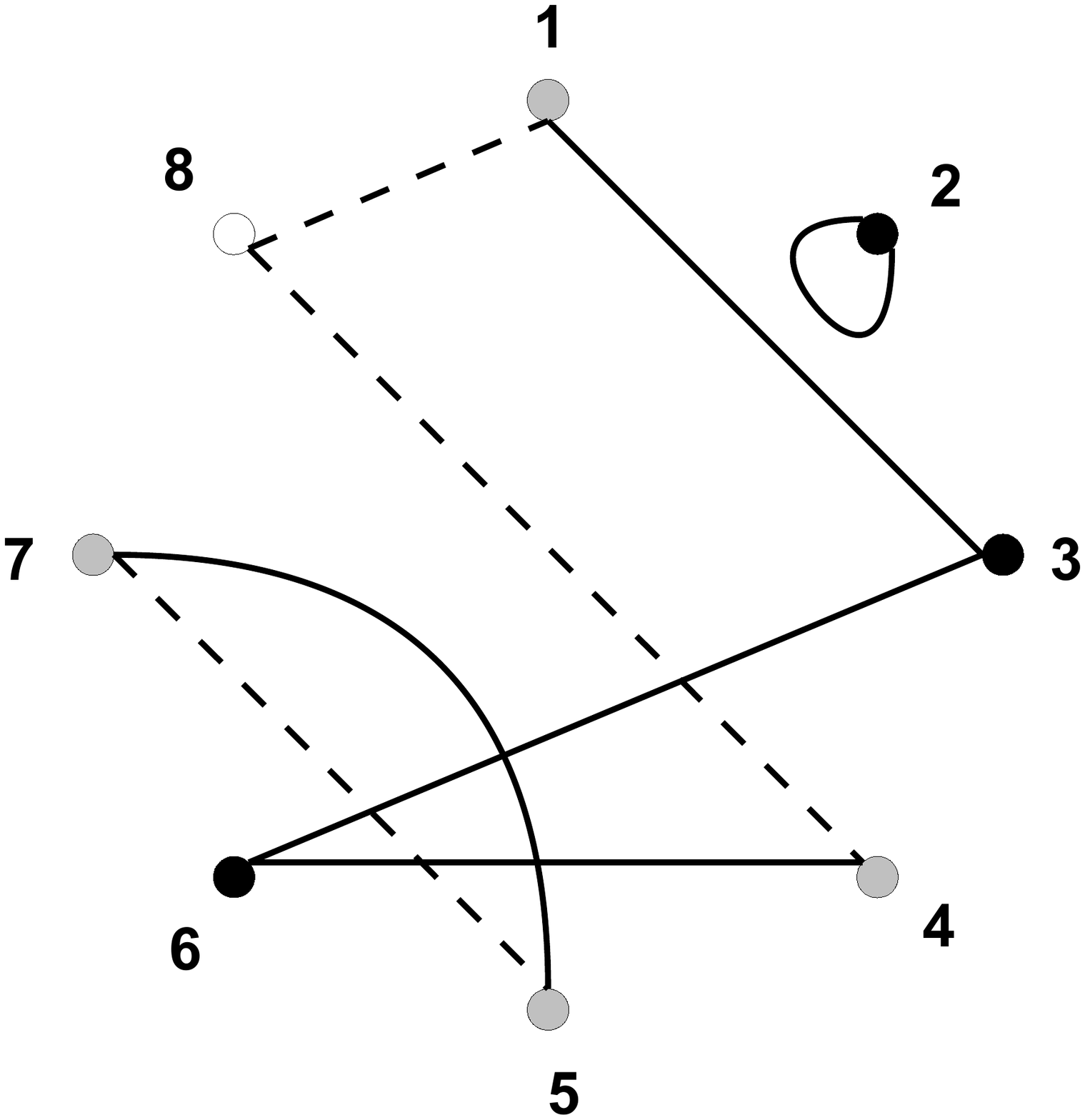}
  \end{minipage}
  \caption{Here, we take $\pi\in\mathcal{P}(8)$ with $B_\pi^{(1)} = \{1,2,3,5,6\}$ and $B_\pi^{(2)} = \{4,7,8\}$, implying that $\overline{B}_\pi^{(1)} = \{1,2,3,4,5,6,7\}$ and $\overline{B}_\pi^{(2)} = \{1,4,5,7,8\}$. In the first picture, we see those vertices colored black which are only contained in $\overline{B}_\pi^{(1)}$, and those colored white which are only in $\overline{B}_\pi^{(2)}$. If some element is in the intersection of the two sets, we color it grey. For any black or white vertex $l$, we have $m_\pi (l)=2$. If $l$ is grey, we have $m_\pi (l) = 1$. To obtain a $\pi$-consistent graph, we first connect all black and grey vertices with each other, where the number of edges which are connected to any vertex $l$ equals $m_\pi (l)$. This leads, for example, to the graph in the second picture. Similarly, we connect all white and grey vertices respecting their multiplicities. One possibility to do so is shown in the third picture. The edges of the subgraph $G_\pi^{(1)}$ are the continuous lines and those of $G_\pi^{(2)}$ are the dashed lines.}
  \label{graphs}
\end{figure}

\begin{rems}
\begin{enumerate}
	\item The same graph can be consistent with different partitions.
	\item Any vertex in a consistent graph has degree two. This implies that any connected component of a consistent graph is a cycle.
	\item Any graph $G$ with vertex set $V(G)=\left\{1,\ldots,k\right\}$ and edge set $E(G)$ such that each vertex has degree $2$ is consistent with the partition $\pi$ of $\left\{1,\ldots,k\right\}$ with only one block.
	\item Both ends of any edge $e\in E(G)$ are in the same set $\overline{B}_\pi^{(s)}$ for some $s=1,\ldots,\#\pi$, that is $f_G(e)\subseteq \overline{B}_\pi^{(s)}$.
	\item If $G$ is consistent with some $\pi\in\mathcal{P}(k)$, then the subgraphs $G_\pi^{(s)}$, $s=1,\ldots,\#\pi$, are uniquely determined apart from their order.
\end{enumerate}
\end{rems}

We denote by $\mathcal{I}(\pi)$ the set of all graphs consistent with the partition $\pi\in\mathcal{P}(k)$. To relate any tuple $\textbf{i}=(i_1,\ldots,i_k)\in\left\{1,\ldots,m\right\}^k$ to a graph $G\in\mathcal{I}(\pi)$ which reflects the structure of $\textbf{i}$ within the closed blocks of $\pi$, we proceed as follows: \\

\begin{enumerate}
	\item For any $s=1,\ldots,\#\pi$, denote by
\begin{equation*}
	r_s:=\sum_{l\in\overline{B}_\pi^{(s)}} m_\pi(l) = 2 \ \# B_\pi^{(s)}
\end{equation*}
the number of elements in $\overline{B}_\pi^{(s)}$ counted with their multiplicities. Now sort the elements $i_l$, $l\in\overline{B}_\pi^{(s)}$, in decreasing order to obtain a vector $x^{(s)}=(l_1,\ldots,l_{r_s})$ with $i_{l_1}\geq \ldots\geq i_{l_{r_s}}$, where each element $i_l$ is supposed to occur exactly $m_\pi(l)$ times. If $i_l=i_{l'}$ for some $l,l'\in\overline{B}_\pi^{(s)}$, we use the convention that the lower index comes first.
	\item[]
	\item Construct a graph $G$ with vertex set $V(G)=\left\{1,\ldots,k\right\}$ by drawing edges between the vertices $x^{(s)}_{2l-1}$ and $x^{(s)}_{2l}$ for any $s=1,\ldots,\#\pi$ and any $l=1,\ldots,\frac{r_s}{2}$.
\end{enumerate}

\grab

To see that $G$ is indeed an element of $\mathcal{I}(\pi)$, we define for any $s=1,\ldots,\#\pi$ the subgraph $G_\pi^{(s)}$ of $G$ to be the graph with vertex set $\overline{B}_\pi^{(s)}$ and edge set $E_\pi^{(s)}(G)$ induced by the vector $x^{(s)}$. These subgraphs obviously fit in the situation of Definition~\ref{def}. \\

Now fix any $G\in\mathcal{I}(\pi)$ and let $\mathcal{T}_m (\pi,G)$ denote the set of all tuples $\textbf{i}=(i_1,\ldots,i_k)\in\left\{1,\ldots,m\right\}^k$ which induce $G$ in the construction above. Equation \eqref{moments} thus becomes
\begin{multline}
\lim_{n\to\infty} \frac 1 m \E\left[\tr\left(\W_n^k\right)\right] \\
 = \lim_{n\to\infty} \sum_{\pi\in\mathcal{P}(k)} \frac{1}{m n^{k-\#\pi}} \sum_{G\in\mathcal{I}(\pi)} \sum_{(i_1,\ldots,i_k) \in \mathcal{T}_m (\pi,G)} \prod_{s=1}^{\#\pi} \E\bigg[\prod_{l\in \overline{B}_\pi^{(s)}} a(i_l,1)^{m_\pi(l)}\bigg],
\label{lim1}
\end{multline}

if the limits exist.

\subsection{Reduction of the set $\mathcal{I}(\pi)$}
\label{reduction}

We want to eliminate those graphs in $\mathcal{I}(\pi)$, $\pi\in\mathcal{P}(k)$, that do not contribute to the limit in \eqref{lim1}. We will see that these are graphs that do not have sufficiently many components. Therefore, recall that the connected components of $G$ are cycles, and let $1\leq r(G)\leq k$ be the number of such cycles denoted by $G_1,\ldots,G_{r(G)}$. We denote the vertex and the edge set of $G_l$ by $V(G_l)$ and $E(G_l)$, respectively. Defining $\#G_l:=\#V(G_l)=\#E(G_l)$, we can identify $G_l$ with a sequence $(v_1(G_l),\ldots,v_{\#G_l}(G_l))$ of vertices, such that \\

\begin{itemize}
	\item $V(G_l)=\{v_1(G_l),\ldots,v_{\#G_l}(G_l)\}$,
	\item[]
	\item $E(G_l)=\{e_1(G_l),\ldots,e_{\#G_l}(G_l)\}$, with $e_j(G_l)$ connecting the vertices $v_j(G_l)$ and $v_{j+1}(G_l)$, where $\#G_l+1$ is identified with $1$.
\end{itemize}

\grab

Further, for $e\in E(G)$, we denote by $\mu(e)=(\mu_1(e),\mu_2(e))$, $\mu_1(e)\leq \mu_2(e)$, the tuple of vertices connected by $e$. With this notation and assumption (A2), we conclude that for any $(i_1,\ldots,i_k)\in \mathcal{T}_m (\pi,G)$,
\begin{multline*}
\bigg|\prod_{s=1}^{\#\pi}  \E\bigg[\prod_{l\in \overline{B}_\pi^{(s)}} a(i_l,1)^{m_\pi(l)}\bigg]\bigg|
\leq c(k) \ \prod_{s=1}^{\#\pi} \prod_{e\in E_\pi^{(s)}(G)} t'(i_{\mu_1(e)},i_{\mu_2(e)}) \\
 = c(k) \ \prod_{e\in E(G)} t'(i_{\mu_1(e)},i_{\mu_2(e)})
 = c(k) \ \prod_{l=1}^{r(G)} \prod_{j=1}^{\#G_l} t'(i_{v_j(G_l)},i_{v_{j+1}(G_l)}),
\end{multline*}

where $E_\pi^{(s)}(G)$, $s=1,\ldots,\#\pi$, are the edge sets described in Definition~\ref{def}. We thus obtain the estimate
\begin{align*}
\sum_{(i_1,\ldots,i_k) \in \mathcal{T}_m (\pi,G)} \bigg|\prod_{s=1}^{\#\pi} \E\bigg[\prod_{l\in \overline{B}_\pi^{(s)}} a(i_l,1)^{m_\pi(l)}\bigg]\bigg|
& \leq c(k) \ \sum_{i_1,\ldots,i_k=1}^m \prod_{l=1}^{r(G)} \prod_{j=1}^{\#G_l} t'(i_{v_j(G_l)},i_{v_{j+1}(G_l)}) \\
& = c(k) \ \prod_{l=1}^{r(G)} \sum_{i_1,\ldots,i_{\#G_l} =1}^m \prod_{j=1}^{\#G_l} t'(i_j,i_{j+1}),
\end{align*}

where we used the fact that different components have distinct vertex sets for the latter identity. Since $t'(i,i')=\alpha^{|i-i'|}$ decays exponentially, we know that the right hand side in the equation above is of order $\mathcal{O}\left(m^{r(G)}\right)=\mathcal{O}\left(n^{r(G)}\right)$. In particular, any graph $G\in\mathcal{I}(\pi)$ with $r(G)<k-\#\pi+1$ gives negligible contribution to the limit in \eqref{lim1}. The following lemma states that in this case, there are only graphs with $k-\#\pi+1$ components left.

\begin{lemma}
Let $\pi\in\mathcal{P}(k)$ with $q:=\#\pi\leq k$. Any graph $G\in\mathcal{I}(\pi)$ has at most $k-q+1$ connected components.
\label{maximum}
\end{lemma}

\begin{proof}
First note that for any $A\subset \left\{1,\ldots,q\right\}$ with $\#A<q$, we can conclude that $\left(\bigcup_{s\in A} \overline{B}_\pi^{(s)}\right)\cap\left(\bigcup_{s\in A^c} \overline{B}_\pi^{(s)}\right) \neq \emptyset$. Indeed, an empty intersection implies
\begin{equation*}
l\in \bigcup_{s\in A} \overline{B}_\pi^{(s)} \ \Longrightarrow \ l-1\in \bigcup_{s\in A} B_\pi^{(s)}.
\end{equation*}
But in this case, the identity $\bigcup_{s\in A} B_\pi^{(s)} = \left\{1,\ldots,k\right\}$ holds, being a contradiction to $\#A<q=\#\pi$. \\
Now consider a graph $G\in\mathcal{I}(\pi)$ with the subgraphs $G_\pi^{(1)},\ldots,G_\pi^{(q)}$ as in Definition~\ref{def}. Recall that each subgraph $G_\pi^{(s)}$ has vertex set $\overline{B}_\pi^{(s)}$ and edge set $E_\pi^{(s)}(G)$ with $\# E_\pi^{(s)}(G) = \#B_\pi^{(s)}$. Start with $s=1$. The maximum number of different components the vertices of $G_\pi^{(1)}$ belong to equals the number of edges in $G_\pi^{(1)}$. If $q=1$, the proof is finished since we have at most $\#B_\pi^{(1)} = k$ components. Otherwise take some different subgraph $G_\pi^{(s)}$, $s=2,\ldots,q$, that has a common vertex with $G_\pi^{(1)}$. Our considerations at the beginning of the proof with $A=\{1\}$ ensure that this choice is possible. Thus, $G_\pi^{(s)}$ induces at most $\#B_\pi^{(s)}-1$ new components in $G$. Proceeding now by taking in each step a subgraph that has a common vertex with one of the subgraphs already considered, we see that the number of components of $G$ does not exceed
\begin{equation*}
\#B_\pi^{(1)} + \sum_{s=2}^{q} \left(\#B_\pi^{(s)}-1\right) = k-q+1.
\end{equation*}
This completes the proof.

\end{proof}

\begin{remark}
The proof of Lemma~\ref{maximum} makes obvious that if we consider some graph $G\in\mathcal{I}(\pi)$ with $k-\#\pi+1$ components, then for any $s=1,\ldots,\#\pi$, all edges of $G_\pi^{(s)}$ must belong to different components in $G$.
\label{samecomp}
\end{remark}

Let $\mathcal{I}^*(\pi)$ denote the set of all graphs $G\in\mathcal{I}(\pi)$ with $r(G)=k-\#\pi+1$. Our considerations above allow us to conclude that

\begin{multline}
\lim_{n\to\infty} \frac 1 m \E\left[\tr\left(\W_n^k\right)\right] \\
= \lim_{n\to\infty} \sum_{\pi\in\mathcal{P}(k)} \frac{1}{m n^{k-\#\pi}} \sum_{G\in\mathcal{I}^*(\pi)} \sum_{(i_1,\ldots,i_k) \in \mathcal{T}_m (\pi,G)} \prod_{s=1}^{\#\pi} \E\bigg[\prod_{l\in \overline{B}_\pi^{(s)}} a(i_l,1)^{m_\pi(l)}\bigg],
\label{limit3}
\end{multline}

if the limits exist.

\subsection{Combinatorial results on the sets $\mathcal{I}^*(\pi)$}

In order to calculate the limit in \eqref{limit3}, we need to deal with the sets $\mathcal{I}^*(\pi)$ in more detail. Therefore, it will prove useful to distinguish between crossing and non-crossing partitions. A partition $\pi\in\mathcal{P}(k)$ is called \emph{crossing} if there are elements $i<j<i'<j'$ such that $i\sim_\pi i'$ and $j\sim_\pi j'$, but $i,i'$ are not in the same block as $j,j'$. Otherwise, $\pi$ is said to be \emph{non-crossing}. We denote the set of all non-crossing partitions of $\{1,\ldots,k\}$ by $\on{NC}(1,\ldots,k)=\on{NC}(k)$. For any $\pi\in\on{NC}(k)$, we will resort to the notion of the \emph{Kreweras complement} $K(\pi)\in\on{NC}(k)$ (cf. \cite{NicaSpeicher}, Definition 9.21). To define $K(\pi)$, consider the numbers $1,\ldots,k$ and $\bar{1},\ldots,\bar{k}$. We interlace them in the alternating way $1 \ \bar{1} \ 2 \ \bar{2} \ldots k \ \bar{k}$. Then, $K(\pi)\in\on{NC}(\bar{1},\ldots,\bar{k})\cong \on{NC}(k)$ is the partition with the following two properties:

\grab

\begin{enumerate}
	\item If $\pi\cup K(\pi)$ is the partition of $\{1,\bar{1}, 2, \bar{2}, \ldots, k, \bar{k}\}$ consisting of the blocks $B_\pi^{(1)},\ldots,B_\pi^{(\#\pi)}, B_{K(\pi)}^{(1)},\ldots,B_{K(\pi)}^{(\#K(\pi))}$, then $\pi\cup K(\pi)\in\on{NC}(1,\bar{1}, 2, \bar{2}, \ldots, k, \bar{k})$;
	\item[]
	\item if a further partition $\sigma\in\on{NC}(\bar{1},\ldots,\bar{k})$ satisfies the property $\pi\cup\sigma\in\on{NC}(1,\bar{1}, 2, \bar{2}, \ldots, k, \bar{k})$, then $K(\pi)$ is bigger than $\sigma$ in the sense that any block of $\sigma$ is contained in some block of $K(\pi)$.
\end{enumerate}

\grab

The construction of the Kreweras complement is illustrated in Figure \ref{kreweras}. Since $K^{2k}(\pi)=\pi$, we can conclude that the map $K:\on{NC}(k)\rightarrow\on{NC}(k)$, $\pi\mapsto K(\pi)$, is a bijection. A further property we will resort to is the equality $\# K(\pi)=k-\#\pi+1$ (for more details see \cite{NicaSpeicher}). The following lemma reveals the structure of the sets $\mathcal{I}^*(\pi)$, and puts them into context with the Kreweras complement.

\begin{figure}[ht]
  \centering
  \begin{minipage}[b]{5.5cm}
    \includegraphics[trim = 0cm 5.5cm 0cm 0cm, width=55mm]{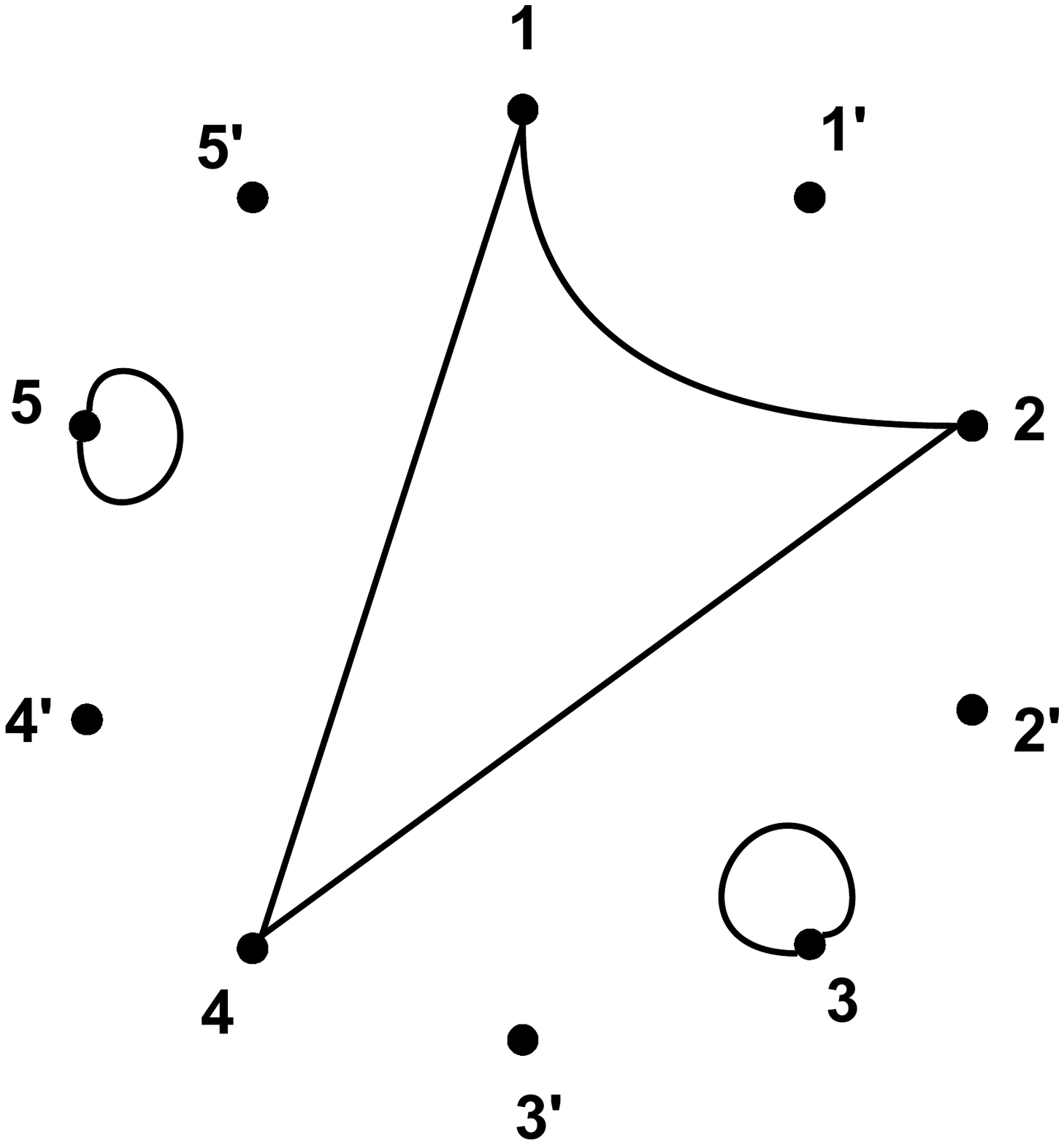}
  \end{minipage}
  \begin{minipage}[b]{5.5cm}
    \includegraphics[trim = 0cm 5.5cm 0cm 0cm, width=55mm]{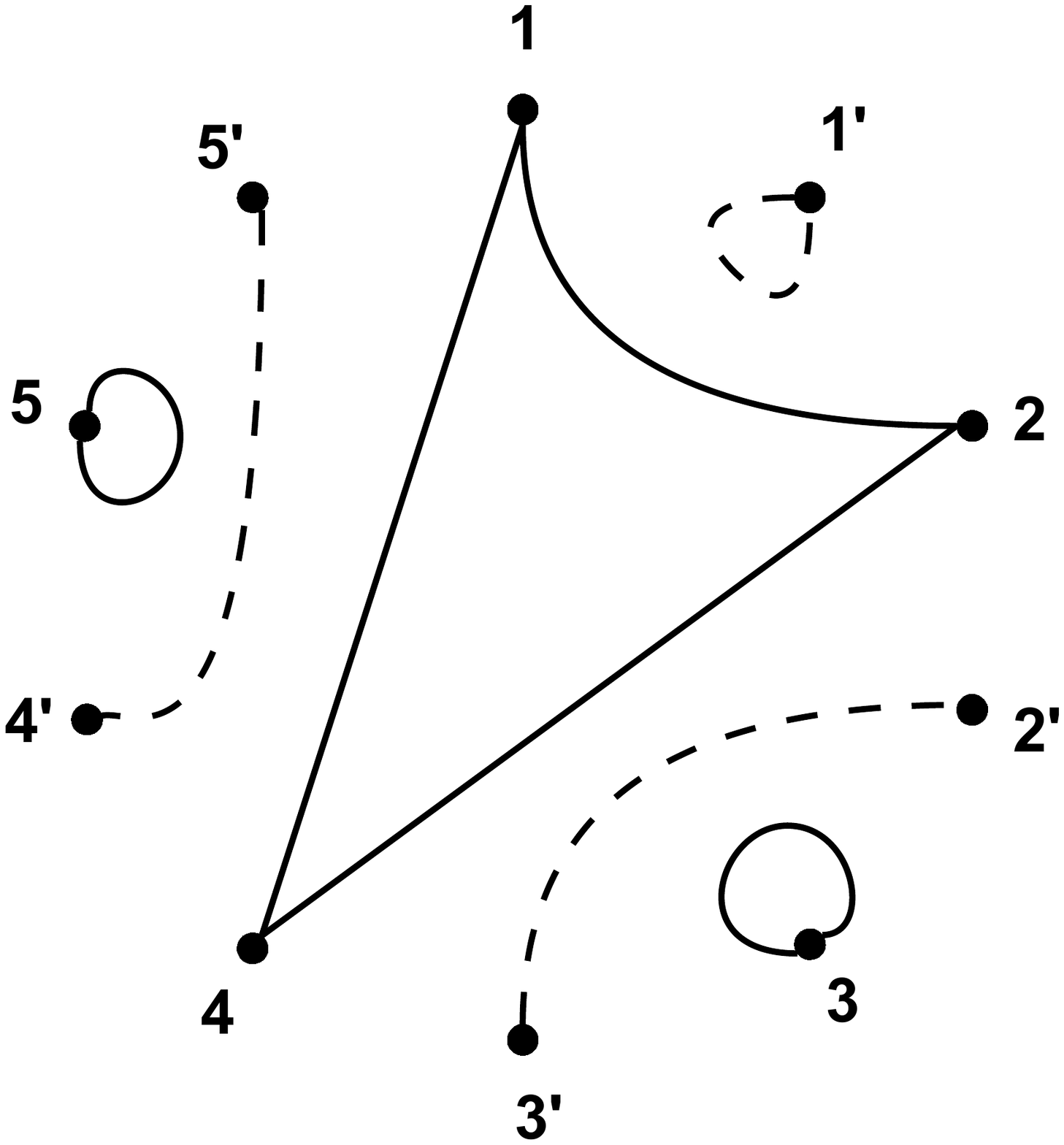}
  \end{minipage}
  \caption{Here, we take $k=5$, and $\pi=\{\{1,2,4\},\{3\},\{5\}\}$. In the first circle, we connect the elements that are in the same block of $\pi$, and draw the vertices $1',\ldots,5'$. The Kreweras complement is then obtained by connecting as many primed vertices as possible without crossing any line. This yields $K(\pi)=\{\{1'\},\{2',3'\},\{4',5'\}\}$ as indicated by the dashed lines in the second circle.}
  \label{kreweras}
\end{figure}

\begin{lemma}
Let $k\in\N$ and $\pi\in\mathcal{P}(k)$.

\begin{enumerate}
	\item If $\pi$ is crossing, we have $\mathcal{I}^*(\pi)=\emptyset$.
	\item If $\pi$ is non-crossing, then $\#\mathcal{I}^*(\pi) = 1$. Moreover, $\pi$ is the Kreweras complement of the partition induced by the unique graph $G\in\mathcal{I}^*(\pi)$ by taking the components as blocks.
\end{enumerate}
\label{ncr}
\end{lemma}

\begin{proof}
We want to verify both statements by induction over $k\in\N$. Therefore, it will be necessary to relate any partition $\pi\in\mathcal{P}(k)$ to some partition in $\mathcal{P}(k-1)$. To this end, we define for any $l=1,\ldots,k$ a partition $\pi_l\in\mathcal{P}(k-1)$ obtained from $\pi$ by first deleting $l$ and then, relabeling any $j\in\{l+1,\ldots,k\}$ to $j-1$. We now want to assign a graph $G_l\in\mathcal{I}(\pi_l)$ to $G\in\mathcal{I}(\pi)$. Thus, suppose that $l\in B_\pi^{(s)}$ for some $s=1,\ldots,\#\pi$. In particular, we have that $l+1\in \overline{B}^{(s)}_\pi$. By definition of consistent graphs, $l$ and $l+1$ are connected to vertices $p$ and $q$, respectively, which are both elements of $\overline{B}^{(s)}_\pi$. Hence, we can substitute the edges $\{l,p\}$ and $\{l+1,q\}$ by $\{l,l+1\}$ and $\{p,q\}$ to obtain a graph which is still an element of $\mathcal{I}(\pi)$. Now eliminate the edge $\{l,l+1\}$ and take the remaining edge connected to $l$ and link it to $l+1$, instead. Consequently, $l$ has become an isolated vertex which can be erased. After relabeling the vertices as before, we finally get a graph $G_l\in\mathcal{I}(\pi_l)$. We now distinguish three cases: \\

\emph{Case 1: $\{l\}$ is a singleton of $\pi$.} In this case, we obtain that $\overline{B}^{(s)}_\pi = \{l,l+1\}$. In particular, there is an edge between $l$ and $l+1$ in any graph $G\in\mathcal{I}(\pi)$. The procedure above then guarantees that $G_l$ and $G$ have an equal number of components. Since $\#\pi_l = \#\pi-1$ and thus, $k-\#\pi+1 = (k-1)-\#\pi_l+1$, we further obtain that $G_l\in\mathcal{I}^*(\pi_l)$ if and only if $G\in\mathcal{I}^*(\pi)$. \\

\emph{Case 2: $l\sim_\pi l-1$.} Assume that $G\in\mathcal{I}^*(\pi)$. Then, the vertex $l$ necessarily has a self-connecting edge. Thus, $G_l$ has one component less than $G$. In this case, the fact that $\#\pi_l = \#\pi$ and $(k-\#\pi+1)-1 = (k-1)-\#\pi_l+1$, implies that $G_l\in\mathcal{I}^*(\pi_l)$ if and only if $G\in\mathcal{I}^*(\pi)$. \\

\emph{Case 3: $\{l\}$ is not a singleton of $\pi$ and $l\not\sim_\pi l-1$.} If this is the case, then it might happen that we merge two different components when substituting the edges in $G$. Hence, $G_l$ has either as much components as $G$ or one less. \\

We want to remark that in case 1 and 2, we have a bijection between $\mathcal{I}^*(\pi)$ and $\mathcal{I}^*(\pi_l)$. Indeed, this can be explained by the reversibility of the construction of $G_l$. The simpler case is the second one where we only need to re-insert the vertex $l$ and draw a self-connecting edge. In the first case, $l$ has to be re-inserted, too. However, we have to argue why there is a unique way to connect $l$ to $G_l$ which leads to the graph $G$. Therefore, assume without loss of generality that $B_{\pi_l}^{(1)}$ denotes the block of $l$ in $\pi_l$. Since $G_l$ is a $\pi_l$-consistent graph, we know that $l$ is connected to some $j\in \overline{B}_{\pi_l}^{(1)}$. Further, $l$ is connected to a second vertex $p$. By Remark~\ref{samecomp}, the fact that $G_l$ has the maximum number of components implies that either $p=j$ or $p\not\in\overline{B}_{\pi_l}^{(1)}$. Consequently, the choice of $p$ is unique. To obtain $G$ from $G_l$, we thus need to introduce a new vertex $l'$. Since $l'$ is supposed to adopt the role of the vertex following $l-1$, we substitute the edge $\{l,p\}$ by $\{l',p\}$. Now relabel all vertices $q\in\{l,\ldots,k-1\}$ to $q+1$ and $l'$ to $l$. Due to the fact that the block $\{l\}$ is a singleton in $\pi$, we are forced to draw an edge between $l$ and $l+1$. But then we get exactly the graph $G$. \\

(i) Suppose that $\pi\in\mathcal{P}(k)$ is crossing. Since a crossing partition requires at least four elements, we start with $k=4$. In this case, the only crossing partition $\pi\in\mathcal{P}(4)$ is given by $\pi=\{\{1,3\},\{2,4\}\}$, implying that $\overline{B}_\pi^{(1)}=\overline{B}_\pi^{(2)}=\{1,2,3,4\}$. Hence, a graph $G\in\mathcal{I}(\pi)$ has at most $2$ components. However, $G\in\mathcal{I}^*(\pi)$ is supposed to have $k-\#\pi+1=3$ components which allows us to conclude that, indeed, $\mathcal{I}^*(\pi)=\emptyset$.

Now take any $k\in\N$ and a crossing partition $\pi\in\mathcal{P}(k)$. Fix some $l\in\{1,\ldots,k\}$ such that $\pi_l$ is still crossing. This choice is possible whenever $k\geq 5$. Further, suppose that there is some $G\in\mathcal{I}^*(\pi)$, \ie $G$ has $k-\#\pi+1$ components. Our construction above yields that in case 1 and 2, $G_l\in\mathcal{I}^*(\pi_l)$, which is contradictory to the inductional hypothesis. Assuming that case holds, we obtain that $G_l$ has at least $k-\#\pi=(k-1)-\#\pi_l+1$ components. This is also impossible. To sum up, we have $\mathcal{I}^*(\pi)=\emptyset$. \\

(ii) Suppose that $\pi\in\mathcal{P}(k)$ is non-crossing. Here, we need to start with $k=1$, implying that $\pi=\{\{1\}\}$. Obviously, there is exactly one possibility to obtain a consistent graph $G$, which can be realized by drawing a self-connecting edge for vertex $1$. In particular, we have one component. This is the maximum number in this case and we see that $\mathcal{I}^*(\pi)$ contains exactly one element. The partition induced by $G$ is equal to $\pi$, and $K(\pi)=\pi$ since $k=1$. Hence, (ii) holds.

Take any $k\in\N$ and assume that $\pi\in\mathcal{P}(k)$ is non-crossing. Then, we can find some $l\in\{1,\ldots,k\}$ such that either $\{l\}$ is a singleton, or $l\sim_\pi l-1$. These are the cases 1 and 2 in the construction at the beginning of the proof. Since they ensured a bijection between $\mathcal{I}^*(\pi)$ and $\mathcal{I}^*(\pi_l)$, the inductional hypothesis yields $\#\mathcal{I}^*(\pi)=\#\mathcal{I}^*(\pi_l)=1$. It remains to verify that $\pi$ is the Kreweras complement of the partition induced by the unique graph $G\in\mathcal{I}^*(\pi)$. We will denote this partition by $\eta$. Let $\eta_l$ be the partition induced by $G_l$. In particular, we have $K(\eta_l)=\pi_l$. First assume that $\{l\}$ is a singleton in $\pi$. In this case, $\{l,l+1\}$ is a block of $\eta$. As a consequence, $K(\eta)$ is the partition obtained from $K(\eta_l)=\pi_l$ after relabeling and inserting the block $\{l\}$. This is exactly the procedure how to re-construct $\pi$ from $\pi_l$, implying $K(\eta)=\pi$. Now suppose that $l\sim_\pi l-1$. Here, $\{l\}$ is a singleton in $\eta$. Hence, we obtain $K(\eta)$ from $K(\eta_l)$ if we relabel the elements $q\in\{l,\ldots,k-1\}$ to $q+1$, and add the element $l$ to the block of $l-1$. Again, the resulting partition is equal to $\pi$. This proves (ii).

\end{proof}

For $\pi\in \mathrm{NC}(k)$, let $G(\pi)$ denote the unique element in $\mathcal{I}^*(\pi)$, and define $\mathcal{T}_m (\pi):=\mathcal{T}_m (\pi,G(\pi))$. Lemma~\ref{ncr} now entails the relation
\begin{multline}
\lim_{n\to\infty} \frac 1 m \E\left[\tr\left(\W_n^k\right)\right] \\
 = \lim_{n\to\infty} \sum_{\pi\in \mathrm{NC}(k)} \frac{1}{m n^{k-\#\pi}} \sum_{(i_1,\ldots,i_k) \in \mathcal{T}_m (\pi)} \ \prod_{s=1}^{\#\pi} \E\bigg[\prod_{l\in \overline{B}_\pi^{(s)}} a(i_l,1)^{m_\pi(l)}\bigg],
\label{limit4}
\end{multline}

if the limits exist.

\subsection{A representation of the joint moments}

Fix $\pi\in \mathrm{NC}(k)$ and $(i_1,\ldots,i_k) \in \mathcal{T}_m (\pi)$. For convenience, we put $G:=G(\pi)$. Recall the subgraphs $G^{(s)}:=G_\pi^{(s)}$, $s=1,\ldots,\#\pi$, as introduced in Definition~\ref{def}. Denote by $E^{(s)}(G):=\{e_1^{(s)},\ldots,e_{\# B_\pi^{(s)}}^{(s)}\}$ the set of the edges of $G^{(s)}$. To finally apply formula (A1) in our context, we define for any $l=1,\ldots,\# B_\pi^{(s)}$
\begin{align*}
i_{2l-1}^{(s)} := i_{\mu_1\left(e_l^{(s)}\right)}, \qquad i_{2l}^{(s)} := i_{\mu_2\left(e_l^{(s)}\right)}.
\end{align*}

We then have
\begin{align*}
\prod_{s=1}^{\#\pi}  \E\bigg[\prod_{l\in \overline{B}_\pi^{(s)}} a(i_l,1)^{m_\pi(l)}\bigg]
 = \prod_{s=1}^{\#\pi} \ \bigg( \prod_{l=1}^{\# B_\pi^{(s)}} t\left(i_{2l-1}^{(s)} ,i_{2l}^{(s)}\right) + R\left(i_1^{(s)},\ldots ,i_{2\# B_\pi^{(s)}}^{(s)} \right)\bigg),
\end{align*}

Recall that we denoted by $G_1,\ldots,G_{k-\#\pi+1}$ the components of $G$, and the edges of $G_l$ by $\{v_j(G_l),v_{j+1}(G_l)\}$, $j=1,\ldots,\#G_l$. With this notation, we obtain
\begin{equation*}
\prod_{s=1}^{\#\pi}  \E\bigg[\prod_{l\in \overline{B}_\pi^{(s)}} a(i_l,1)^{m_\pi(l)}\bigg]
= \prod_{l=1}^{k-\#\pi+1} \prod_{j=1}^{\#G_l} t(i_{v_j(G_l)},i_{v_{j+1}(G_l)}) + R'\left(i_1,\ldots,i_{k}\right),
\end{equation*}

where
\begin{equation*}
R'\left(i_1,\ldots,i_{k}\right)
= \sum_{\substack{A\subseteq \{1,\ldots,\#\pi\}, \\ A\neq \emptyset}} \prod_{s\in A} R\left(i_1^{(s)},\ldots ,i_{2\# B_\pi^{(s)}}^{(s)} \right) \prod_{s\in A^c} \prod_{l=1}^{\# B_\pi^{(s)}} t\left(i_{2l-1}^{(s)} ,i_{2l}^{(s)}\right).
\end{equation*}

We want to verify that $R'\left(i_1,\ldots,i_{k}\right)$ gives negligible contribution to the limit. Therefore, fix $A\subseteq \{1,\ldots,\#\pi\}$, $A\neq \emptyset$. Without loss of generality, we assume that $1\in A$. Taking account of the identity in (A1) and the estimate in (A2), we conclude that for any $s=1,\ldots,\#\pi$,
\begin{equation*}
\left|R\left(i_1^{(s)},\ldots ,i_{2\# B_\pi^{(s)}}^{(s)}\right)\right|
\leq c(k) \prod_{l=1}^{\# B_\pi^{(s)}} t'\left(i_{2l-1}^{(s)} ,i_{2l}^{(s)}\right).
\end{equation*}

Hence,
\begin{align*}
&\prod_{s\in A} \left|R\left(i_1^{(s)},\ldots ,i_{2\# B_\pi^{(s)}}^{(s)}\right)\right| \ \prod_{s\in A^c} \prod_{l=1}^{\# B_\pi^{(s)}} \left|t\left(i_{2l-1}^{(s)} ,i_{2l}^{(s)}\right)\right| \\
&\qquad\leq c(k) \ \left|R\left(i_1^{(1)},\ldots ,i_{2\# B_\pi^{(1)}}^{(1)}\right)\right| \
\prod_{s=2}^{\#\pi} \ \prod_{l=1}^{\# B_\pi^{(s)}} t'\left(i_{2l-1}^{(s)} ,i_{2l}^{(s)}\right) \\
&\qquad\leq c(k) \ \sum_{j=1}^{\# B_\pi^{(1)}} \sum_{\sigma\in S_{2\# B_\pi^{(1)}}^j} \prod_{l=1}^j t'\left(i_{\sigma(2l-1)}^{(1)},i_{\sigma(2l)}^{(1)}\right) \prod_{s=2}^{\#\pi} \ \prod_{l=1}^{\# B_\pi^{(s)}} t'\left(i_{2l-1}^{(s)} ,i_{2l}^{(s)}\right).
\end{align*}

Now fix $j\in\{1,\ldots,\# B_\pi^{(1)}\}$ and $\sigma\in S_{2\# B_\pi^{(1)}}^j$. Define a graph $G'$ obtained from $G$ by deleting the edges in $E^{(1)}(G)$. Since all edges in $E^{(1)}(G)$ belonged to different components by Remark~\ref{samecomp}, $G'$ has still $k-\#\pi+1$ components. Denote by $G'_j$ the graph obtained from $G'$ by inserting the edges
\begin{equation*}
\left\{\mu_{\sigma(2l-1)}\left(e_{\left\lceil \sigma(2l-1)/2\right\rceil}^{(1)}\right),\mu_{\sigma(2l)}\left(e_{\left\lceil \sigma(2l)/2\right\rceil}^{(1)}\right)\right\}, \qquad l=1,\ldots,j,
\end{equation*}

where $\mu_{\sigma(l)}:=\mu_1$ if $\sigma(l)$ is odd, and otherwise $\mu_{\sigma(l)}:=\mu_2$. The definition of the set $S_{2\# B_\pi^{(1)}}^j$ entails that there is at least one edge in $G_j'$ which was not contained in $G$. Hence, at least two components of $G'$ are merged when constructing $G_j'$. Consequently, $G_j'$ has $d\leq k-\#\pi$ components which are either complete cycles or open cycles, \ie one edge is missing. Denote those components by $G_{j,1}',\ldots,G_{j,d}'$. We now have
\begin{equation*}
\prod_{l=1}^j t'\left(i_{\sigma(2l-1)}^{(1)},i_{\sigma(2l)}^{(1)}\right) \prod_{s=2}^{\#\pi} \ \prod_{l=1}^{\# B_\pi^{(s)}} t'\left(i_{2l-1}^{(s)} ,i_{2l}^{(s)}\right)
= \prod_{l=1}^d \ \prod_{e\in E(G_{j,l}')} t'(i_{\mu_1(e)},i_{\mu_2(e)}),
\end{equation*}

and
\begin{equation*}
\sum_{(i_1,\ldots,i_k) \in \mathcal{T}_m (\pi)} \prod_{l=1}^d \ \prod_{e\in E(G_{j,l}')} t'(i_{\mu_1(e)},i_{\mu_2(e)})
\leq \prod_{l=1}^d \ \sum_{i_1,\ldots,i_{\# G_{j,l}'}=1}^m \ \prod_{e\in E(G_{j,l}')} t'(i_{\mu_1(e)},i_{\mu_2(e)}).
\end{equation*}

The exponential form of $t'(i,j)$ ensures that the right hand side is of order $m^d$. Consequently,
\begin{equation*}
\sum_{(i_1,\ldots,i_k) \in \mathcal{T}_m (\pi)} R'\left(i_1,\ldots,i_{k}\right) = \mathcal{O}(m^{k-\#\pi}),
\end{equation*}

implying
\begin{multline}
\lim_{n\to\infty} \frac 1 m \E\left[\tr\left(\W_n^k\right)\right] = \\
\lim_{n\to\infty} \sum_{\pi\in \mathrm{NC}(k)} \frac{1}{m n^{k-\#\pi}} \sum_{(i_1,\ldots,i_k) \in \mathcal{T}_m (\pi)} \prod_{l=1}^{k-\#\pi+1} \prod_{j=1}^{\#G_l(\pi)} t(i_{v_j(G_l(\pi))},i_{v_{j+1}(G_l(\pi))}),
\label{eq2}
\end{multline}

if the limits exist.

\subsection{An extension of the set $\mathcal{T}_m (\pi)$}

If we summed over all elements $(i_1,\ldots,i_k)\in\{1,\ldots,m\}^k$ instead of over all tuples in $\mathcal{T}_m (\pi)$, we would obtain traces of powers of $\T_m$ on the right hand side of \eqref{eq2} which would enable us to calculate the limit. The next lemma  gives us the justification to do so.

\begin{lemma}
We have
\begin{equation*}
\begin{split}
& \lim_{n\to\infty}\frac 1 m \E\left[\tr\left(\W_n^k\right)\right] \\
& \qquad = \lim_{n\to\infty} \sum_{\pi\in \mathrm{NC}(k)} \frac{1}{m n^{k-\#\pi}} \sum_{i_1,\ldots,i_k=1}^m \prod_{l=1}^{k-\#\pi+1} \ \prod_{j=1}^{\#G_l(\pi)} t(i_{v_j(G_l(\pi))},i_{v_{j+1}(G_l(\pi))}) \\
& \qquad = \lim_{n\to\infty} \sum_{\pi\in \mathrm{NC}(k)} \frac{1}{m n^{k-\#\pi}} \prod_{l=1}^{k-\#\pi+1} \tr\left(\T_m^{\#G_l(\pi)}\right).
\end{split}
\end{equation*}
\label{fin}
\end{lemma}

\begin{proof}
We have already seen how the second equality can be obtained, \eg in Section \ref{reduction}. To see that the first equality holds, let $\pi\in \mathrm{NC}(k)$ and put $G:=G(\pi)$. We want to verify that $\left(\mathcal{T}_m (\pi)\right)^c$ gives negligible contribution to the limit in \eqref{eq2}. Thus fix $(i_1,\ldots,i_k) \in \left(\mathcal{T}_m (\pi)\right)^c$. Being not an element of the set $\mathcal{T}_m (\pi)$ means by definition that there is an $s=1,\ldots,\#\pi$ and two edges $e, e' \in E(G)$ such that the ends $\mu_1(e), \mu_2(e)$ and $\mu_1(e'), \mu_2(e')$ are elements of $\overline{B}_\pi^{(s)}$ but $i_{\mu_1(e)}\leq i_{\mu_1(e')}\leq i_{\mu_2(e)}\leq i_{\mu_2(e')}$. We then have
\begin{align}
\begin{split}
t'(i_{\mu_1(e)},i_{\mu_2(e)}) \ t'(i_{\mu_1(e')},i_{\mu_2(e')})
& = \alpha^{i_{\mu_2(e)}-i_{\mu_1(e)}+i_{\mu_2(e')}-i_{\mu_1(e')}} \\
& = t'(i_{\mu_1(e)},i_{\mu_2(e')}) \ t'(i_{\mu_1(e')},i_{\mu_2(e)}).
\end{split}
\label{trel}
\end{align}

Since $G$ has the maximum number of components, Remark~\ref{samecomp} yields that $e$ and $e'$ belong to different components of $G$. Consequently, substituting the edges $e$ and $e'$ by the edges $\left\{{\mu_1(e'),\mu_2(e)}\right\}$ and $\left\{{\mu_1(e),\mu_2(e')}\right\}$  leads to a new graph $G^{(e,e')}\in\mathcal{I}(\pi)$ that has $r(G^{(e,e')})=k-\#\pi$ components. These considerations and relation \eqref{trel} imply
\begin{equation*}
\prod_{l=1}^{k-\#\pi+1} \ \prod_{j=1}^{\#G_l} t'(i_{v_j(G_l)},i_{v_{j+1}(G_l)})
= \prod_{l=1}^{k-\#\pi} \ \prod_{j=1}^{\#G^{(e,e')}_l} t'\left(i_{v_j\left(G^{(e,e')}_l\right)},i_{v_{j+1}\left(G^{(e,e')}_l\right)}\right).
\end{equation*}

Further, by assumption (A2), we know that $|t(i,i')|\leq c t'(i,i')$ for any $i,i'\in\N$. Hence we obtain
\begin{multline*}
\sum_{(i_1,\ldots,i_k) \in \left(\mathcal{T}_m (\pi)\right)^c} \prod_{l=1}^{k-\#\pi+1} \ \prod_{j=1}^{\#G_l} \left|t\left(i_{v_j(G_l)},i_{v_{j+1}(G_l)}\right)\right| \\
\leq \sum_{e,e'\in E(G)} \sum_{i_1,\ldots,i_k=1}^m \prod_{l=1}^{k-\#\pi} \ \prod_{j=1}^{\#G^{(e,e')}_l} t'\left(i_{v_j\left(G^{(e,e')}_l\right)},i_{v_{j+1}\left(G^{(e,e')}_l\right)}\right) = \mathcal{O} \left(m^{k-\#\pi}\right),
\end{multline*}
and, consequently,
\begin{equation*}
\lim_{n\to\infty} \sum_{\pi\in \mathrm{NC}(\pi)} \frac{1}{m n^{k-\#\pi}} \sum_{(i_1,\ldots,i_k) \in (\mathcal{T}_m (\pi))^c} \prod_{l=1}^{k-\#\pi+1} \ \prod_{j=1}^{\#G_l(\pi)} t\left(i_{v_j(G_l(\pi))},i_{v_{j+1}(G_l(\pi))}\right) = 0.
\end{equation*}

This completes the proof.
\end{proof}

Lemma~\ref{fin} and relation \eqref{mt} now yield
\begin{equation}
\begin{split}
\lim_{n\to\infty}\frac 1 n \E\left[\tr\left(\W_n^k\right)\right]
&= \lim_{n\to\infty} \sum_{\pi\in \mathrm{NC}(k)} \ \left(\frac{m}{n}\right)^{k-\#\pi} \ \prod_{l=1}^{k-\#\pi+1} \frac{1}{m} \ \tr\left(\T_m^{\#G_l(\pi)}\right) \\
&= \sum_{\pi\in \mathrm{NC}(k)} \ y^{k-\#\pi} \ \prod_{l=1}^{k-\#\pi+1} H_{\#G_l(\pi)}.
\end{split}
\label{limit6}
\end{equation}

\subsection{Non-crossing partitions}

Recall that for any $\pi\in \mathrm{NC}(k)$, Lemma~\ref{ncr} provided us with the identity $K(\eta)=\pi$, if $\eta$ denotes the partition induced by $G(\pi)$. Since $K:\mathrm{NC}(k)\rightarrow\mathrm{NC}(k)$ is a bijection, and $\# K(\eta)=k-\#\eta+1$, we find that
\begin{equation*}
\sum_{\pi\in \mathrm{NC}(k)} \ y^{k-\#\pi} \ \prod_{l=1}^{k-\#\pi+1} H_{\#G_l(\pi)} = \sum_{\pi\in \mathrm{NC}(k)} \ y^{\#\pi-1} \ \prod_{l=1}^{\#\pi} \ H_{\# B_{\pi}^{(s)}}.
\end{equation*}

To finally obtain the representation in Theorem~\ref{main}, we want to sort the partitions in $\mathrm{NC}(k)$ by the number and the size of their blocks. Thus, define for any $i_1,\ldots,i_k\in\{1,\ldots,k\}$ the set $\mathrm{NC}(k;i_1,\ldots,i_k)$ of all partitions in $\mathrm{NC}(k)$ with $i_l$ blocks of size $l$, $l=1,\ldots,k$. Note that if some partition $\pi\in \mathrm{NC}(k;i_1,\ldots,i_k)$ has $q=k-s+1$ blocks, then $i_{s+1}=\ldots=i_k=0$. Indeed, if $\pi$ had a block consisting of at least $s+1$ elements, we would have at most $k-(s+1)$ elements left to form the remaining $q-1=k-s$ blocks. Hence, we will write $\mathrm{NC}(k;i_1,\ldots,i_s)$ instead of $\mathrm{NC}(k;i_1,\ldots,i_k)$ whenever $\#\pi=k-s+1$. Note that $i_1+ \ldots +i_s$ is the number of blocks and $i_1+2i_2+ \ldots +si_s$ the total number of elements of any partition $\pi\in \mathrm{NC}(k;i_1,\ldots,i_s)$. Thus we consider only tuples $(i_1,\ldots,i_s)\in\{0,\ldots,k-s+1\}^s$ satisfying $i_1+ \ldots +i_s=k-s+1$ and $i_1+2i_2+ \ldots +si_s=k$. Now we get
\begin{equation*}
\sum_{\pi\in \mathrm{NC}(k)} \ y^{\#\pi-1} \ \prod_{l=1}^{\#\pi} \ H_{\# B_{\pi}^{(s)}}
= \sum_{s=1}^k \ y^{k-s} \ \sum_{\substack{i_1+\ldots + i_s = k-s+1, \\ i_1 +2i_2+\ldots + si_s=k}} \# \mathrm{NC}(k;i_1,\ldots,i_s) \ \prod_{l=1}^{s} H_{l}^{i_l}.
\end{equation*}

It remains to determine $\# \mathrm{NC}(k;i_1,\ldots,i_s)$. This can be achieved with the help of

\begin{lemma}[\cite{Kreweras72}, Theorem 4]
The number of non-crossing partitions $\pi\in\mathcal{P}(k)$ with $i_l$ blocks of size $l$, $l=1,\ldots,k$, equals
\begin{equation*}
\frac{k!}{(k-q+1)!}\ \frac{1}{i_1!\cdot\ldots\cdot i_k!},
\end{equation*}
where $q:=\#\pi = i_1+\ldots+i_k$.
\label{noncr}
\end{lemma}

We thus obtain that
\begin{equation*}
\lim_{n\to\infty}\frac 1 m \E\left[\tr\left(\W_n^k\right)\right] = \sum_{s=1}^{k} \ y^{k-s} \ \frac{k!}{s!} \ \sum_{\substack{i_1+\ldots + i_s = k-s+1, \\ i_1 +2i_2+\ldots + si_s=k}} \ \prod_{l=1}^{s} \frac{H_{l}^{i_l}}{i_l!},
\end{equation*}

which is exactly the statement of Theorem~\ref{main}.

\bibliography{bib}

\begin{thebibliography}{BYK86}

\bibitem[AZ08]{zeitouni}
G.~W. Anderson and O.~Zeitouni.
\newblock A {CLT} for regularized sample covariance matrices.
\newblock {\em Ann. Statist.}, 36(6):2553--2576, 2008.

\bibitem[Bai99]{bai_survey}
Z.~D. Bai.
\newblock Methodologies in spectral analysis of large-dimensional random
  matrices, a review.
\newblock {\em Statist. Sinica}, 9(3):611--677, 1999.
\newblock With comments by G. J. Rodgers and Jack W. Silverstein; and a
  rejoinder by the author.

\bibitem[BS10]{BaiSilv}
Z.~Bai and J.~W. Silverstein.
\newblock {\em Spectral analysis of large dimensional random matrices}.
\newblock Springer Series in Statistics. Springer, New York, 2010.

\bibitem[BYK86]{baiyinkrish}
Z.~D. Bai, Y.~Q. Yin, and P.~R. Krishnaiah.
\newblock On limiting spectral distribution of product of two random matrices
  when the underlying distribution is isotropic.
\newblock {\em J. Multivariate Anal.}, 19(1):189--200, 1986.

\bibitem[BZ08]{BaiZhou08}
Z.~Bai and W.~Zhou.
\newblock Large sample covariance matrices without independence structures in
  columns.
\newblock {\em Statistica Sinica}, 18:425--442, 2008.

\bibitem[Geo88]{Georgii}
H.-O. Georgii.
\newblock {\em Gibbs Measures and Phase Transitions}.
\newblock De Gruyter Studies in Mathematics 9. Walter de Gruyter, Berlin, 1988.

\bibitem[Iss18]{Isserlis}
L.~Isserlis.
\newblock On a formula for the product-moment coefficient of any order of a
  normal frequency distribution in any number of variables.
\newblock {\em Biometrika}, 12(1/2):134--139, 1918.

\bibitem[Kre72]{Kreweras72}
G.~Kreweras.
\newblock Sur les partitions non crois\'{e}es d'un cycle.
\newblock {\em Discrete Math.}, 1(4):333--350, 1972.

\bibitem[MP67]{pastur_marcenko}
V.~A. Mar{\v{c}}enko and L.~A. Pastur.
\newblock Distribution of eigenvalues in certain sets of random matrices.
\newblock {\em Mat. Sb. (N.S.)}, 72 (114):507--536, 1967.

\bibitem[NS06]{NicaSpeicher}
A.~Nica and R.~Speicher.
\newblock {\em Lectures on the Combinatorics of Free Probability}.
\newblock London Mathematical Society Lecture Note Series 335. Cambridge
  University Press, Cambridge, 2006.

\bibitem[Sil95]{Silverstein}
J.~W. Silverstein.
\newblock Strong convergence of the empirical distribution of eigenvalues of
  large dimensional random matrices.
\newblock {\em Journal of Multivariate Analysis}, 55(2):331--339, 1995.

\bibitem[Wac78]{wachter}
Kenneth~W. Wachter.
\newblock The strong limits of random matrix spectra for sample matrices of
  independent elements.
\newblock {\em Ann. Probability}, 6(1):1--18, 1978.

\bibitem[Yin86]{yin}
Y.~Q. Yin.
\newblock Limiting spectral distribution for a class of random matrices.
\newblock {\em J. Multivariate Anal.}, 20(1):50--68, 1986.

\bibitem[YK85]{yinkrish}
Y.~Q. Yin and P.~R. Krishnaiah.
\newblock Limit theorem for the eigenvalues of the sample covariance matrix
  when the underlying distribution is isotropic.
\newblock {\em Teor. Veroyatnost. i Primenen.}, 30(4):810--816, 1985.

\end{thebibliography}

\end{document}